\documentclass[11pt]{article}
\usepackage{cite}
\usepackage{amsmath,amsthm,amssymb,amsfonts}
\usepackage{a4,latexsym,epsfig}
\usepackage{color,colordvi,amscd}

\newtheorem{theorem}{Theorem}[section]
\newtheorem{proposition}[theorem]{Proposition}%[section]

\newtheorem{lemma}[theorem]{Lemma}%[section]
\newtheorem{corollary}[theorem]{Corollary}%[section]
\title{Spectral conditions for graph rigidity in the Euclidean plane}
\author{Sebastian M. Cioab\u{a}\thanks{
Department of Mathematical Sciences,
University of Delaware, Newark, DE 19716, USA. E-mail: {\tt cioaba@udel.edu}} ,\,
Sean Dewar\thanks{
Johann Radon Institute for Computational and Applied Mathematics (RICAM), Austrian Academy of Sciences, 4040 Linz, Austria. E-mail: {\tt sean.dewar@ricam.oeaw.ac.at}}\,  and 
Xiaofeng Gu\thanks{
Department of Mathematics, University of West Georgia, Carrollton, GA 30118, USA. E-mail: {\tt xgu@westga.edu}; Corresponding author.}
}

\begin{document}
\date{\today}
\maketitle
\noindent

\begin{abstract}
Rigidity is the property of a structure that does not flex. It is well studied in discrete geometry and mechanics, and has applications in material science, engineering and biological sciences. A bar-and-joint framework is a pair $(G,p)$ of graph $G$ together with a map $p$ of the vertices of $G$ into the Euclidean plane. We view the edges of $(G, p)$ as bars and the vertices as universal joints. The vertices can move continuously as long as the distances between pairs of adjacent vertices are preserved. The framework is rigid if any such motion preserves the distances between all pairs of vertices. In 1970, Laman obtained a combinatorial characterization of rigid graphs in the Euclidean plane. In 1982, Lov\'asz and Yemini discovered a new characterization and proved that every $6$-connected graph is rigid. Combined with a characterization of global rigidity, their proof actually implies that every 6-connected graph is globally rigid. Consequently, if Fiedler's algebraic connectivity is greater than 5, then $G$ is globally rigid. In this paper, we improve this bound and show that for a graph $G$ with minimum degree $\delta\geq 6$, if its algebraic connectivity is greater than $2+\frac{1}{\delta -1}$, then $G$ is rigid and if its algebraic connectivity is greater than $2+\frac{2}{\delta -1}$, then $G$ is globally rigid. Our results imply that every connected regular Ramanujan graph with degree at least $8$ is globally rigid. We also prove a more general result giving a sufficient spectral condition for the existence of $k$ edge-disjoint spanning rigid subgraphs. The same condition implies that a graph contains $k$ edge-disjoint spanning $2$-connected subgraphs. This result extends previous spectral conditions for packing edge-disjoint spanning trees.
\end{abstract}



\section{Introduction}

In this paper, we consider finite undirected simple graphs. Throughout the paper,
$k$ denotes a positive integer and $G$ denotes a simple graph with vertex set $V(G)$
and edge set $E(G)$. 

Rigidity is the property of a structure that does not flex. Arising from mechanics, rigidity has been studied in discrete geometry and combinatorics 
(see, among others, \cite{GrSS93, Jord16}) and has applications in material science, engineering and biological sciences (see \cite{CSC,GC,ZFBG} for example). A {\bf $d$-dimensional framework} is a pair $(G, p)$, where $G$ is a graph and
$p$ is a map from $V(G)$ to $\mathbb{R}^d$. Roughly speaking, it is a straight line realization
of $G$ in $\mathbb{R}^d$. Two frameworks $(G, p)$ and $(G, q)$ are {\bf equivalent} if
$\|p(u) - p(v) \| = \|q(u) - q(v) \|$ holds for every edge $uv\in E(G)$, where $\|\cdot \|$
denotes the Euclidean norm in $\mathbb{R}^d$. Two frameworks $(G, p)$ and $(G, q)$ 
are {\bf congruent} if $\|p(u) - p(v) \| = \|q(u) - q(v) \|$ holds for every $u, v\in V(G)$.
A framework $(G,p)$ is {\bf generic} if the coordinates of its points are algebraically independent over the rationals. 
The framework $(G, p)$ is {\bf rigid} if there exists an $\varepsilon >0$ such that
if $(G, p)$ is equivalent to $(G, q)$ and $\|p(u) - q(u) \| < \varepsilon$ for every $u\in V(G)$, then $(G, p)$ is congruent to $(G, q)$. 
As observed in \cite{AsRo78}, a generic realization of $G$ is rigid in $\mathbb{R}^d$ if and only if every generic 
realization of $G$ is rigid in $\mathbb{R}^d$.  Hence the generic rigidity can be considered as a property of the underlying graph.
A graph is called {\bf rigid} in $\mathbb{R}^d$ if every/some generic realization of $G$ is rigid in $\mathbb{R}^d$.

A $d$-dimensional framework $(G, p)$ is {\bf globally rigid} if every framework that is equivalent to $(G, p)$ is congruent to $(G, p)$. This can be considered as the framework being rigid with $\epsilon$ being chosen to be as large as you wish. In \cite{GortlerThurstonHealey10} it was proven that if there exists a generic framework $(G,p)$ in $\mathbb{R}^d$ that is globally rigid, then any other generic framework $(G,q)$ in $\mathbb{R}^d$ will also be globally rigid. Following from this, we define a graph $G$ to be {\bf globally rigid} in $\mathbb{R}^d$ if there exists a globally rigid generic framework $(G,p)$ in $\mathbb{R}^d$. A closely related concept to global rigidity is redundant rigidity. A graph $G$ is {\bf redundantly rigid} in $\mathbb{R}^d$ if $G-e$ is rigid in $\mathbb{R}^d$ for every edge $e \in E(G)$. It was proven by Hendrickson \cite{Hendrickson92} that any globally rigid graph in $\mathbb{R}^d$ is $(d+1)$-connected and redundantly rigid in $\mathbb{R}^d$. Hendrickson \cite{Hendrickson92} also conjectured the converse. It can be shown easily that it is true for $d=1$, however the conjecture is not true for $d \geq 3$ \cite{ConnWhit10}. The final case $d=2$ of the conjecture was confirmed to be true by \cite{Conn05} and \cite{JaJo05}. 

In the following, we will focus on rigidity and global rigidity only in $\mathbb{R}^2$, unless otherwise stated.
For a subset $X\subseteq V(G)$, let $G[X]$ be the subgraph of $G$ induced by $X$ and $E(X)$ denote the edge set of $G[X]$.
A graph $G$ is {\bf sparse} if $|E(X)|\le 2|X|-3$ for every $X\subseteq V(G)$
with $|X|\ge 2$. By definition, any sparse graph is simple. If in addition $|E(G)|=2|V(G)|-3$, then $G$ is called {\bf $(2, 3)$-tight}. 
Laman \cite{Lama70} proved that a graph $G$ is rigid in $\mathbb{R}^2$ if $G$ contains a spanning $(2, 3)$-tight subgraph. 
Thus a $(2, 3)$-tight graph is also called a {\bf minimally rigid} graph. In history, it was first discovered by Pollaczek-Geiringer \cite{Poll1927,Poll1932} who made notable progress on the properties of minimally rigid graphs, but her work was forgotten. Laman \cite{Lama70} rediscovered the characterization of minimally rigid graphs in $\mathbb{R}^2$. Since then, a minimally rigid graph is known as a {\bf Laman graph}.

Lov\'{a}sz and Yemini \cite{LoYe82} gave a new characterization of rigid graphs using matroid theory and showed that $6$-connected graphs are rigid. 
In fact they proved that every $6$-connected graph is rigid even with the removal of any three edges.
It implies that $6$-connected graphs are redundantly rigid and thus globally rigid. 
They also constructed infinitely many $5$-connected graphs that are not rigid. 

We will study rigidity and global rigidity from spectral graph theory viewpoint. We describe the matrices and the eigenvalues of our interest below. If $G$ is an undirected simple graph with $V(G)=\{v_1,v_2,\cdots,v_n\}$, its {\bf adjacency matrix} is the $n$ by $n$ matrix $A(G)$ with entries $a_{ij}=1$ 
if there is an edge between $v_i$ and $v_j$ and $a_{ij}=0$ otherwise, for $1\le i,j\le n$. Let $D(G) = (d_{ij})_{1\leq i,j\leq n}$ be the {\bf degree matrix} of $G$, that is, the $n$ by $n$ diagonal matrix with $d_{ii}$ being the degree of vertex $v_i$ in $G$ for $1\le i\le n$. The matrix $L(G)=D(G)-A(G)$ is called the {\bf Laplacian matrix} of $G$. For $1\leq i\leq n$, we use $\mu_i(G)$ to denote the $i$-th smallest eigenvalue of $L(G)$. It is not difficult to see that $\mu_1(G)=0$. The second smallest eigenvalue of $L(G)$, $\mu_2(G)$, is known as the {\bf algebraic connectivity} of $G$. 

Fiedler \cite{Fiedler} proved that the vertex-connectivity of $G$ is at least $\mu_2(G)$. Thus, the theorem of Lov\'asz and Yemini \cite{LoYe82} 
that every $6$-connected graph is (globally) rigid implies that if $\mu_2(G) > 5$, then $G$ is (globally) rigid. 
In this paper, we will improve this sufficient condition to ``$\mu_{2}(G)> 2+ \frac{1}{\delta -1}$'' for rigidity (Corollary~\ref{maincor})
and  to ``$\mu_{2}(G)> 2+ \frac{2}{\delta -1}$'' for global rigidity (Corollary~\ref{glob}).
Actually we obtain more general sufficient spectral conditions for the existence of $k$ edge-disjoint 
spanning rigid subgraphs in Theorem~\ref{eigkrig} and Corollary~\ref{kdisrig}. 

\begin{theorem}\label{eigkrig}
Let $G$ be a graph with minimum degree $\delta(G)\ge 6k$. If 
\\(1) $\displaystyle\mu_{2}(G)>\frac{6k-1}{\delta(G)+1}$,
\\(2) $\displaystyle\mu_{2}(G-u)>\frac{4k-1}{\delta(G-u)+1}$ for every $u\in V(G)$, and
\\(3) $\displaystyle\mu_{2}(G-v-w)>\frac{2k-1}{\delta(G-v-w)+1}$ for every $v,w\in V(G)$,
\\then $G$ contains at least $k$ edge-disjoint spanning rigid subgraphs.
\end{theorem}

Theorem~\ref{eigkrig} has the following weaker, but neater, corollary.

\begin{corollary}
\label{kdisrig}
Let $G$ be a graph with minimum degree $\delta \ge 6k$. If $$\mu_{2}(G)> 2+ \frac{2k-1}{\delta -1},$$
then $G$ contains at least $k$ edge-disjoint spanning rigid subgraphs.
\end{corollary}

When $k=1$, we obtain the following spectral condition for rigid graphs.
\begin{corollary}
\label{strcor}
Let $G$ be a graph with minimum degree $\delta(G)\ge 6$. If 
\\(1) $\displaystyle\mu_{2}(G)>\frac{5}{\delta(G)+1}$,
\\(2) $\displaystyle\mu_{2}(G-u)>\frac{3}{\delta(G-u)+1}$ for every $u\in V(G)$, and
\\(3) $\displaystyle\mu_{2}(G-v-w)>\frac{1}{\delta(G-v-w)+1}$ for every $v,w\in V(G)$,
\\then $G$ is rigid.
\end{corollary}
This result is similar in spirit and motivated by the work of Jackson and Jord\'an~\cite{JaJo09}, in which they proved that
a simple graph $G$ is (globally) rigid if $G$ is $6$-edge-connected, $G-u$ is $4$-edge-connected 
for every vertex $u$ and $G-\{v, w\}$ is $2$-edge-connected for any vertices $v,w\in V(G)$. Corollary~\ref{strcor} involves several conditions and we can show that the condition ``$\mu_{2}(G)>\frac{5}{\delta(G)+1}$'' is essentially best possible. A family of examples will be constructed in a later section.

As before, we can also obtain the following weaker, but easier to state and verify, condition for a graph to be rigid. 
\begin{corollary}
\label{maincor}
Let $G$ be a graph with minimum degree $\delta\ge 6$. If $$\mu_{2}(G)> 2+ \frac{1}{\delta -1},$$
then $G$ is rigid.
\end{corollary}

Using the same method as used in Theorem~\ref{eigkrig} when $k=1$, we can prove the following similar results for redundant rigidity and global rigidity.
\begin{theorem}
\label{redund}
Let $G$ be a graph with minimum degree $\delta(G)\ge 6$. If 
\\(1) $\displaystyle\mu_{2}(G)>\frac{6}{\delta(G)+1}$,
\\(2) $\displaystyle\mu_{2}(G-u)>\frac{4}{\delta(G-u)+1}$ for every $u\in V(G)$, and
\\(3) $\displaystyle\mu_{2}(G-v-w)>\frac{2}{\delta(G-v-w)+1}$ for every $v,w\in V(G)$,
\\then $G$ is redundantly rigid.
\end{theorem}
\begin{corollary}
\label{glob}
Let $G$ be a graph with minimum degree $\delta\ge 6$. If $$\mu_{2}(G)> 2+ \frac{2}{\delta -1},$$
then $G$ is globally rigid.
\end{corollary}

In the next section, we will present some preliminaries that will be used in our proofs. The proofs of the main results will be presented in Section 3. 
A family of examples will be constructed in Section 4 to show the best possible bound of $\mu_2(G)$ in Corollary~\ref{strcor}.
In the last section, we will make some concluding remarks and give some applications of our results on rigidity and global rigidity of pseudo-random graphs and Ramanujan graphs, as well as a spectral sufficient condition for the existence of edge-disjoint spanning 2-connected graphs. Some questions are posted.

\section{Preliminaries}
 
The following theorem obtained by Lov\'asz and Yemini \cite{LoYe82} plays a very important role in graph rigidity in $\mathbb{R}^2$.
%which also gives the rank function of rigidity matroid of a graph. 
\begin{theorem}[Lov\'asz and Yemini \cite{LoYe82}]\label{LYthm}
A graph $G$ is rigid if and only if $\sum_{X\in\mathcal{G}}(2|V(X)|-3)\ge 2|V(G)| -3$ 
for every collection $\mathcal{G}$ of induced subgraphs of $G$ whose edges partition $E(G)$.
\end{theorem}

A characterization of global rigidity in $\mathbb{R}^2$ came from the combination of a result of Connelly~\cite{Conn05} and a result of
Jackson and Jord\'an~\cite{JaJo05}.
\begin{theorem}[Connelly~\cite{Conn05}, Jackson and Jord\'an~\cite{JaJo05}]
\label{globthm}
A graph $G$ is globally rigid if and only if $G$ is 3-connected and redundantly rigid, or $G$ is a complete graph on at most three vertices.
\end{theorem}

Sufficient conditions for the existence of edge-disjoint spanning rigid subgraphs have also been well studied. 
Jord\'an \cite{Jord05} showed that every $6k$-connected graph contains $k$ edge-disjoint spanning rigid subgraphs. Cheriyan, Durand de Gevigney and Szigeti \cite{ChGS14} proved that a simple graph $G$ contains $k$ edge-disjoint spanning rigid subgraphs if $G-Z$ is $(6k - 2k|Z|)$-edge-connected for every $Z\subset V(G)$. In fact, they proved a stronger result of packing spanning rigid subgraphs and spanning trees in \cite{ChGS14}. The results were extended to a more general case in \cite{Gu17}. Motivated by the spanning tree packing theorem of Nash-Williams~\cite{Nash61} and Tutte~\cite{Tutt61}, the third author \cite{Gu18} obtained a partition condition for packing spanning rigid subgraphs, described below. 

For any partition $\pi$ of $V(G)$, $e_G (\pi)$ denotes the number of edges of $G$ whose ends lie in two different parts of $\pi$. A part of $\pi$ is {\bf trivial} if it consists of a single vertex. Let $Z\subset V(G)$ and $\pi$ be a partition of $V(G-Z)$ with $n_0$ trivial parts $u_1, u_2,\cdots,u_{n_0}$.  We define $n_Z(\pi)$ to be $\sum_{1\le i\le n_0} |Z_i|$ where $Z_i$ is the set of vertices in $Z$ that are adjacent to $u_i$ for $1\le i\le n_0$. If $Z=\emptyset$, then $n_Z(\pi) =0$.
\begin{theorem}[Gu \hspace{1sp}\cite{Gu18}]\label{parthm}
A graph $G$ contains $k$ edge-disjoint spanning rigid subgraphs if for every $Z\subset V(G)$
and every partition $\pi$ of $V(G-Z)$ with $n_0$ trivial parts and $n'_0$ nontrivial parts, 
$$e_{G-Z} (\pi)\ge k(3-|Z|)n'_0 + 2k n_0 - 3k -n_Z(\pi).$$
\end{theorem}

Now we introduce some useful tools for Laplacian eigenvalues. Fiedler \cite{Fiedler} applied Cauchy interlacing to the Laplacian matrix and obtained the following result (see also \cite[Section 1.7]{BrHa12} and \cite[Thm. 13.5.1]{GoRo01}).
\begin{theorem}[Fiedler \cite{Fiedler}]
\label{mu2int}
If $S$ is a subset of vertices of the graph $G$, then $$\mu_2(G)\le \mu_2(G-S) + |S|.$$
\end{theorem}

For any subset $U\subset V(G)$, $\partial_G (U)$ or simply $\partial (U)$ denotes the set of edges in $G$, 
each of which has one end in $U$ and the other end in $V(G)\backslash U$. We will also need the following result.
\begin{lemma}[Liu et al. \hspace{1sp}\cite{LHGL14}]
\label{LHlem}
Suppose that $X, Y\subset V(G)$ with $X\cap Y =\emptyset$. 
Let $e(X, Y)$ denote the number of edges with one end in $X$ and the other in $Y$.
If $\mu_{2}(G)\ge \max\{\frac{|\partial(X)|}{|X|}, \frac{|\partial(Y)|}{|Y|}\}$,
then $$[e(X, Y)]^2\ge |X|  |Y| \left(\mu_{2}(G)-\frac{|\partial(X)|}{|X|}\right)\left(\mu_{2}(G)-\frac{|\partial(Y)|}{|Y|}\right).$$
\end{lemma}

The following combinatorial lemma is well known. It was used in \cite{GLLY12,KS06} for example. For the sake of completeness, we include a short proof below.
\begin{lemma}\label{cardi}
Let $G$ be a graph with minimum degree $\delta$ and $U$ be a non-empty proper subset of $V(G)$.
If $|\partial(U)| \le \delta -1$, then $|U| \ge \delta + 1$.
\end{lemma}
\begin{proof}[\bf Proof:]
We argue by contradiction and assume that $|U|\le \delta$. Then $|U|(|U|-1)+ |\partial(U)| \ge |U|\delta$ by counting the
total degrees of vertices in $U$. But $|U|(|U|-1)+|\partial(U)| \le \delta(|U|-1)+(\delta-1)\le |U|\delta-1$,
contrary to the fact that $|U|(|U|-1)+|\partial(U)|\ge |U|\delta$. Thus $|U|\ge\delta+1$.
\end{proof}

%%%%%%%%%%
\section{The proofs of the main results}

In this section, we present the proofs of Theorem~\ref{eigkrig}, Corollary~\ref{kdisrig}, Theorem~\ref{redund} and Corollary~\ref{glob}. We first restate Theorem~\ref{eigkrig} below and present its proof.
\begin{theorem}\label{gzeig}
Let $G$ be a graph with minimum degree $\delta(G)\ge 6k$. If 
$$\mu_{2}(G-Z)>\frac{6k-2k|Z|-1}{\delta(G-Z) +1}$$
for every $Z\subset V(G)$ with $|Z|\le 2$,
then $G$ has at least $k$ edge-disjoint spanning rigid subgraphs.
\end{theorem}
\iffalse%%%%%%
\begin{corollary}
Let $G$ be a graph with minimum degree $\delta(G)\ge 6k$. 
\\(i) If $\lambda_2(G-Z) < \delta(G-Z) - \frac{6k-2k|Z|-1}{\delta(G-Z) +1}$
for every $Z\subset V(G)$ with $|Z|\le 2$,
then $G$ has at least $k$ edge-disjoint spanning rigid subgraphs.
\\(ii) If $q_2(G-Z) < 2\delta(G-Z) -\frac{6k-2k|Z|-1}{\delta(G-Z) +1}$
for every $Z\subset V(G)$ with $|Z|\le 2$,
then $G$ has at least $k$ edge-disjoint spanning rigid subgraphs.
\end{corollary}
\fi%%%%%%

\begin{proof}
Let $H=G-Z$.
By Theorem~\ref{parthm}, it suffices to show that
for  any partition $\pi$ of $V(H)$ with $n_0$ trivial parts and $n'_0$ nontrivial parts, 
\begin{equation}
\label{ndts}
e_{H} (\pi)\ge k(3-|Z|)n'_0 + 2k n_0 - 3k -n_Z(\pi),
\end{equation}
for every $Z\subset V(G)$.

We first prove that if $|Z|\ge 3$, then (\ref{ndts}) is always true. Actually, for every trivial part (a single vertex) $u_j$,
its degree $d_H (u_j)$ in $H$ must satisfy the inequality $d_H (u_j)\ge \delta(G) - |Z_j| \ge 6k - |Z_j|$, where $Z_j$ is the set of neighbors of $u_j$ in $Z$.
Recall that $n_Z(\pi)=\sum_{1\le j\le n_0}|Z_j|$. If $|Z|\ge 3$, then
\begin{eqnarray*}
e_{H} (\pi)
&\ge& \frac{1}{2}\sum_{1\le j\le n_0} d_H (u_j)\\
&\ge& \frac{1}{2}\sum_{1\le j\le n_0} \delta (G) -\frac{1}{2}\sum_{1\le j\le n_0}|Z_j|\\
&\ge& 3kn_0 - \frac{1}{2} n_Z(\pi)\\
&\ge& k(3-|Z|)n'_0 + 2k n_0 - 3k -n_Z(\pi).
\end{eqnarray*}

We assume that $|Z|\le 2$ from now on. If $V_1, V_2, \cdots, V_{n'_0}$ are the nontrivial parts in the partition $\pi$ of $H$ and
$u_1, u_2, \cdots, u_{n_0}$ are the trivial parts of $\pi$, then
\begin{equation}
\label{tripa}
\sum_{1\le j\le n_0} d_H (u_j)\ge \sum_{1\le j\le n_0}( \delta(G) - |Z_j|) \ge 6kn_0 - n_Z(\pi).
\end{equation}

Without loss of generality, we may assume that $|\partial_H (V_1)|\le |\partial_H (V_2)|\le \cdots \le |\partial_H (V_{n'_0})|$. 
For convenience, we will use $\partial$ for $\partial_{H}$ in the following. If $|\partial(V_2)|\ge 6k-2k|Z|$, then
\begin{eqnarray*}
e_{H} (\pi)
& =  &   \frac{1}{2}\left( \sum_{1\le i\le n'_0} |\partial(V_i)| + \sum_{1\le j\le n_0} d_H (u_j)\right)\\
&\ge &  \frac{1}{2}\left( (6k-2k|Z|)(n'_0-1) + 6kn_0 - n_Z(\pi)\right)\\
&\ge &  k(3-|Z|)n'_0 + 2k n_0 - 3k -n_Z(\pi), 
\end{eqnarray*}
and we are done. Thus, we assume that $n'_0 \ge 2$ and $|\partial(V_2)|\le 6k-2k|Z| -1$.

Let $q$ be the largest index such that $|\partial(V_q)|\le 6k-2k|Z| -1$. Then $2\le q\le n'_0$. 
Therefore,
\begin{equation}
\label{igeq1}
|\partial(V_i)|\ge 6k-2k|Z|  \mbox{ for } q < i\le n'_0,
\end{equation}
whenever such an $i$ exists.

For $1\le i\le q$, since $|\partial(V_i)|\le 6k-2k|Z| -1 \le \delta(G) - 2k|Z| -1 \le \delta(G-Z) -1 =  \delta(H) -1$, 
Lemma~\ref{cardi} implies that $|V_i|\ge \delta(H) +1$. 
As $\mu_{2}(H)>\frac{6k-2k|Z|-1}{\delta(H)+1}$, it follows that
$|V_i| \mu_{2}(H) > 6k-2k|Z|-1$ for $1\le i\le q$. By Lemma~\ref{LHlem}, for $2\le i\le q$,
\begin{eqnarray*}
[e(V_1, V_i)]^2
&\ge& |V_1||V_i|\left(\mu_{2}(H)-\frac{|\partial(V_1)|}{|V_1|}\right)\left(\mu_{2}(H)-\frac{|\partial(V_i)|}{|V_i|}\right)\\
&=& \left(|V_1|\mu_{2}(H)- |\partial(V_1)|\right)\left(|V_i|\mu_{2}(H)- |\partial(V_i)|\right)\\
&>& \left(6k-2k|Z|-1- |\partial(V_1)|\right)\left(6k-2k|Z|-1- |\partial(V_i)|\right)\\
&\ge& \left(6k-2k|Z|-1 - |\partial(V_i)|\right)^2.
\end{eqnarray*}
Thus $e(V_1, V_i) > 6k-2k|Z|-1- |\partial(V_i)|$, and so $e(V_1, V_i) \ge 6k-2k|Z|- |\partial(V_i)|$. We get that 
$$|\partial(V_1)|\ge \sum_{2\le i\le q}e(V_1, V_i) \ge (6k-2k|Z|)(q-1)- \sum_{2\le i\le q} |\partial(V_i)|,$$
and thus
\begin{equation}
\label{ileq}
\sum_{1\le i\le q} |\partial(V_i)|= |\partial(V_1)| + \sum_{2\le i\le q} |\partial(V_i)| \ge (6k-2k|Z|)(q-1).
\end{equation}

Using \eqref{tripa}, \eqref{igeq1} and \eqref{ileq}, we obtain that
\begin{eqnarray*}
e_{H} (\pi)
& =  & \frac{1}{2}\left( \sum_{1\le i\le n'_0} |\partial(V_i)| + \sum_{1\le j\le n_0} d_H (u_j) \right)\\
& =  & \frac{1}{2}\left(\sum_{1\le i\le q} |\partial(V_i)| + \sum_{q< i\le n'_0} |\partial(V_i)| + \sum_{1\le j\le n_0} d_H (u_j)\right)\\
&\ge &  \frac{1}{2}\left( (6k-2k|Z|)(q-1) +(6k-2k|Z|)(n'_0-q) + 6kn_0 - n_Z(\pi) \right)\\
&\ge &  k(3-|Z|)n'_0 + 2k n_0 - 3k -n_Z(\pi), 
\end{eqnarray*}
which completes the proof.
\end{proof}

Corollary~\ref{kdisrig} follows directly from Theorem~\ref{eigkrig} and the following lemma.
\begin{lemma}
\label{mu2lem}
Let $G$ be a graph with minimum degree $\delta \ge 6k$. If 
$\mu_{2}(G)> 2+ \frac{2k-1}{\delta -1}$,
then $\mu_{2}(G-Z)>\frac{6k-2k|Z|-1}{\delta(G-Z) +1}$
for every $Z\subset V(G)$ with $|Z|\le 2$.
\end{lemma}
\begin{proof}
Notice that $\delta(G-u)\ge \delta -1$ for every $u\in V(G)$ and $\delta(G-v-w)\ge \delta -2$ for every $v,w\in V(G)$.
It suffices to show that $\mu_{2}(G)>\frac{6k-1}{\delta+1}, \mu_{2}(G-u)>\frac{4k-1}{\delta}$ and $\mu_{2}(G-v-w)>\frac{2k-1}{\delta -1}$. Because $\delta\geq 6k$, it is not hard to verify that $2+ \frac{2k-1}{\delta -1} > \frac{6k-1}{\delta+1}$ and $1+ \frac{2k-1}{\delta -1} > \frac{4k-1}{\delta}$. Thus $\mu_{2}(G)> 2+ \frac{2k-1}{\delta -1} > \frac{6k-1}{\delta+1}$. By Theorem~\ref{mu2int}, $\mu_{2}(G-u)\ge \mu_{2}(G) -1 > 1+ \frac{2k-1}{\delta -1} >  \frac{4k-1}{\delta}$ and $\mu_{2}(G-v-w)\ge \mu_{2}(G) -2 > \frac{2k-1}{\delta -1}$.
\end{proof}

%\bigskip

The proofs of Theorem~\ref{redund} and Corollary~\ref{glob} are quite similar to the above. We first restate Theorem~\ref{redund} as below
and give a quick proof. 
\begin{theorem}\label{l:1}
	Let $G$ be a graph with minimum degree $\delta(G) \geq 6$. If
	\begin{align*}
		\mu_2(G-Z) > \frac{6-2|Z|}{\delta(G-Z) +1}
	\end{align*}
	for every $Z \subset V(G)$ with $|Z| \leq 2$, then $G$ is redundantly rigid.
\end{theorem}
\begin{proof}
We need to show that for any edge $f\in E(G)$, $G-f$ is rigid.
By Theorem~\ref{parthm}, it suffices to show that
for  any partition $\pi$ of $V(G-f-Z)$ with $n_0$ trivial parts and $n'_0$ nontrivial parts, 
\begin{equation*}
e_{G-f-Z} (\pi)\ge (3-|Z|)n'_0 + 2n_0 - 3 -n_Z(\pi),
\end{equation*}
for every $Z\subset V(G)$. If $\pi$ consists of exactly one part, then $n_0 =0$, $n'_0 =1$, $e_{G-f-Z}(\pi)=0$,
and clearly the above inequality holds. Thus we may assume that $\pi$ contains at least two parts in the following.

Notice that $e_{G-f-Z}(\pi) \geq e_{G-Z}(\pi) - 1$, and it suffices to show that
\begin{align}\label{eq:0}
e_{G-Z}(\pi) \geq (3-|Z|)n'_o +2n_0 - 2 - n_Z(\pi).
\end{align}

First, (\ref{eq:0}) will hold if $|Z| \geq 3$. The proof is the same as that of Theorem \ref{gzeig} for $k=1$, and thus will be omitted here.
We assume that $|Z|\le 2$ in the remaining proof.

\noindent
{\bf Case 1:} $n'_0 \leq 1$. As $\delta(G) \geq 6$, we have $2e_{G-Z}(\pi) \geq \delta(G-Z)n_0 \geq (6-|Z|)n_0$.
If (\ref{eq:0}) does not hold then
	\begin{align*}
		(3-|Z|)n'_o +2n_0 - 2 - n_Z(\pi) > e_{G-Z}(\pi) \geq \frac{1}{2}(6-|Z|) n_0,
	\end{align*}
which yields
	\begin{align*}
		(6-2|Z|)n'_0 -4 - 2n_Z(\pi) > (2-|Z|) n_0.
	\end{align*}
Given that $n'_0 \leq 1$ and $|Z| \leq 2$, the above inequality holds only when $|Z|=0$, $n_0 =0$ and $n'_0 =1$. However this implies that $\pi$ consists of exactly one part, violating our assumption. Hence (\ref{eq:0}) must hold.

\noindent
{\bf Case 2:} $n'_0 \geq 2$. This case is similar to the proof of Theorem \ref{gzeig}, and thus will be omitted.
\end{proof}

\begin{lemma}
\label{mu2again}
Let $G$ be a graph with minimum degree $\delta \ge 6k$. If $\mu_{2}(G)> 2+ \frac{2}{\delta -1}$,
then $\mu_{2}(G-Z)>\frac{6-2|Z|}{\delta(G-Z) +1}$ for every $Z\subset V(G)$ with $|Z|\le 2$.
\end{lemma}
\begin{proof}
Notice that $\delta(G-u)\ge \delta -1$ for every $u\in V(G)$ and $\delta(G-v-w)\ge \delta -2$ for every $v,w\in V(G)$.
It suffices to show that $\mu_{2}(G)>\frac{6}{\delta+1}, \mu_{2}(G-u)>\frac{4}{\delta}$ and $\mu_{2}(G-v-w)>\frac{2}{\delta -1}$. 
Because $\delta\geq 6k$, it is not hard to verify that $2+ \frac{2}{\delta -1} > \frac{6}{\delta+1}$ and $1+ \frac{2}{\delta -1} > \frac{4}{\delta}$. 
Thus $\mu_{2}(G)> 2+ \frac{2}{\delta -1} > \frac{6}{\delta+1}$. By Theorem~\ref{mu2int}, $\mu_{2}(G-u)\ge \mu_{2}(G) -1 > 1+ \frac{2}{\delta -1} >  \frac{4}{\delta}$ and $\mu_{2}(G-v-w)\ge \mu_{2}(G) -2 > \frac{2}{\delta -1}$.
\end{proof}

\begin{proof}[The proof of Corollary~\ref{glob}]
By Theorem~\ref{l:1} and Lemma~\ref{mu2again}, $G$ is redundantly rigid. Since $\mu_2(G) >2$, the vertex-connectivity of $G$ is at least 3.
Hence $G$ is globally rigid by Theorem~\ref{globthm}.
\end{proof}

\section{Examples}
In this section, we construct a family of graphs to show that the condition ``$\mu_{2}(G)>\frac{5}{\delta(G)+1}$'' in Corollary~\ref{strcor} is essentially best possible. 

\begin{figure}[htb]
\begin{center}
\includegraphics[width=0.7\textwidth]{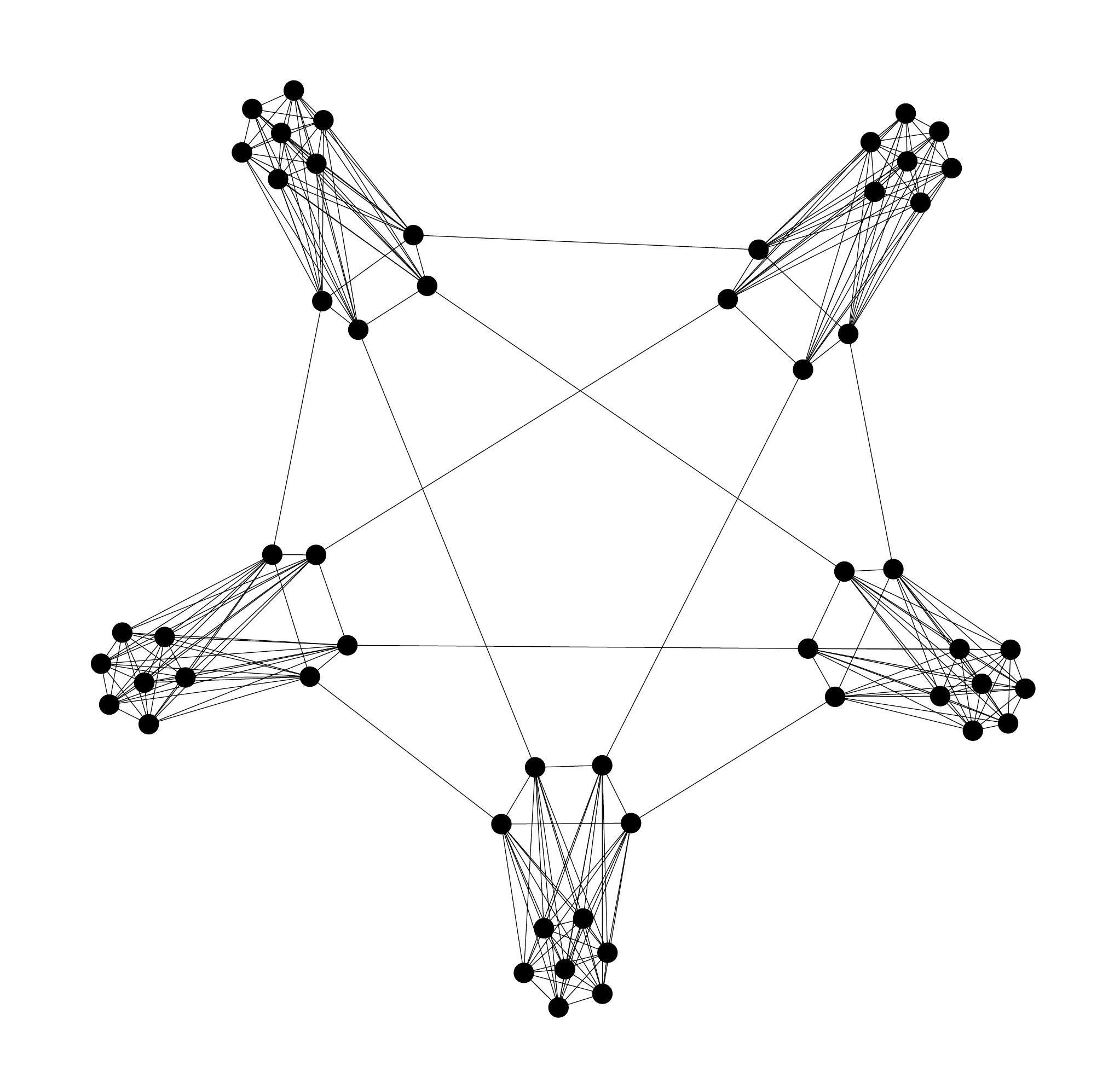}
\caption{An example of $\mathcal{H}_d$ when $d=10$}
\label{h10}
\end{center}
\end{figure}

The family of graphs was initially constructed in \cite{CiWo12}.
Let $d\ge 6$ be an integer and let $H_1, H_2, H_3, H_4, H_5$ be 5 vertex-disjoint copies of a graph obtained
from $K_{d+1}$ by deleting two disjoint edges. Suppose that the deleted edges are $a_i b_i$ and $u_i v_i$
in $H_i$ for $1\le i\le 5$. Let $\mathcal{H}_d$ be the $d$-regular graph whose vertex set is $\bigcup_{i=1}^5 V(H_i)$
and whose edge set is the union of $\bigcup_{i=1}^5 E(H_i)$ with the set
$F=\{b_1a_2, b_2a_3, b_3a_4, b_4a_5, b_5a_1, u_1v_3, u_3v_5, u_5v_2, u_2v_4, u_4v_1\}$.
An example is shown in Figure~\ref{h10} when $d=10$. By the computation in \cite{CiWo12}, 
it follows that $\frac{5}{d+3} < \mu_2(\mathcal{H}_d)\le \frac{5}{d+1}$ for $d\ge 6$.
However, we can show that $\mathcal{H}_d$ is not rigid as below.

Let $X_i = V(H_i)$ for $1\le i\le 5$, and for $6\le i\le 15$, $X_i$ be the vertex set induced by a single edge in $F$.
Let $\mathcal{G} = \{X_i: 1\le i\le 15\}$. Clearly $\{E(X), X\in\mathcal{G}\}$ partitions $E(G)$. Then 
$$\sum_{X\in\mathcal{G}}(2|X|-3) = 5(2(d+1) -3) + 10(2\times 2 -3) =10d +5.$$
Notice that $|V(G)|= 5d + 5$ and it follows that $2|V| -3 = 10d +7$. By Theorem~\ref{LYthm}, $\mathcal{H}_d$ is not rigid.

\section{Concluding remarks}

In this paper, we discovered improved spectral conditions for rigid graphs and globally rigid graphs in $\mathbb{R}^2$ from Laplacian eigenvalues.
Corollaries~\ref{maincor} and \ref{glob} give simple conditions for rigidity and global rigidity, respectively. However we do not know whether they are best possible. Corollary~\ref{glob} might be close to best possible, as a necessary condition for globally rigid graphs in $\mathbb{R}^2$ is 3-connectedness.
It would be interesting to see how large can $\mu_2(G)$ be for non-rigid graphs. Another problem of interest would be obtaining a spectral condition for a graph to contain a spanning $(a,b)$-tight subgraph for other values of $a$ and $b$.

One immediate application of our results is on packing spanning 2-connected subgraphs. Since every rigid graph with at least $3$ vertices is $2$-connected, by Corollary~\ref{kdisrig}, we have the following result on edge-disjoint spanning $2$-connected subgraphs. This result can be seen as a spectral analogue of Jord\'{a}n's combinatorial sufficient condition \cite{Jord05} for packing spanning rigid subgraphs and hence spanning 2-connected subgraphs. It also extends the spectral conditions for vertex-connectivity of \cite{Fiedler, KS06, CiGu16, Gu21}, and the spectral conditions for $k$ edge-disjoint spanning trees of \cite{CiWo12, CiGu16, Gu13, GLLY12, LHGL14, HGLL16}, to $k$ edge-disjoint spanning 2-connected subgraphs.
\begin{proposition}
Let $G$ be a graph with minimum degree $\delta \ge 6k$. If $\mu_{2}(G)> 2+ \frac{2k-1}{\delta -1},$
then $G$ has at least $k$ edge-disjoint spanning 2-connected subgraphs. 
\end{proposition}

We may also consider eigenvalues of other matrices. The matrix $Q(G)=D(G)+A(G)$ is called the {\bf signless Laplacian matrix} of $G$. 
For $1\leq i\leq n$, we use $\lambda_i :=\lambda_i(G)$ and $q_i :=q_i(G)$ to denote the $i$-th largest eigenvalue of $A(G)$ and $Q(G)$, respectively.
By Courant-Weyl inequalities (on page 29 of \cite{BrHa12}), it is not hard to see that $\mu_{2} + \lambda_2 \ge \delta$ and $\delta + \lambda_2\le q_2$.
Thus all results involving $\mu_2$ in the paper will imply sufficient conditions using $\lambda_2$ and $q_2$. For example, by Corollary~\ref{glob}, it follows that for a graph $G$ with minimum degree $\delta\ge 6$, if $\lambda_2(G) < \delta -2 -\frac{2}{\delta -1}$ (or alternatively, if $q_2(G) < 2\delta -2 -\frac{2}{\delta -1}$), then $G$ is globally rigid. Other results of the paper can be modified in similar ways.

One reason to switch to eigenvalues of other matrices stems from pseudo-random graphs. 
Define $\lambda(G) = \max_{2\le i\le n} |\lambda_i (G)| =\max\{|\lambda_2(G)|, |\lambda_n(G)|\}$.
We call $\lambda(G)$ the second largest absolute eigenvalue of $G$. It is known that a $d$-regular graph on $n$ vertices with {\em small} $\lambda(G)$ (compared to $d$; for example, $\lambda(G)=\Theta(\sqrt{d})$) has edge distribution similar to the random graph of same edge density, namely it is a kind of {\bf pseudo-random graph} (see \cite{KS06} for more details). 
Clearly, the results in this paper imply sufficient conditions for the rigidity and global rigidity of pseudo-random graphs. 

A connected $d$-regular graph $G$ is called a {\bf Ramanujan graph} if $|\lambda_i (G)|\leq 2\sqrt{d-1}$ for all $\lambda_i (G)\neq \pm d$ with $2\le i\le n$.
By definition, if $G$ is a $d$-regular Ramanujan graph, then $\mu_2(G) \geq d - 2\sqrt{d-1}$.
The rigidity of Ramanujan graphs was investigated by Servatius \cite{Serv}. By Corollary~\ref{glob}, we can conclude that any connected $d$-regular Ramanujan graph with $d\geq 8$ is globally rigid in $\mathbb{R}^2$. It is unknown whether this holds for smaller values of $d$, although as all cubic graphs (with the exception of the complete graph on four vertices) are flexible, we know that the lowest value of $d$ where it could hold is either $4,5,6$ or $7$.
We plan to focus on these specific graphs in more detail in future work \cite{CDG2}.

\par\bigskip
\noindent
{\bf Acknowledgments}
\\The authors would like to thank the anonymous referees for reviewing the manuscript and providing insightful comments.
The first author is supported by NSF grants DMS-1600768, CIF-1815922 and a JSPS Invitational Fellowship for Research in Japan S19016.
The second author is supported by Austrian Science Fund (FWF): P31888.
The third author is supported by a grant from the Simons Foundation (522728, XG).

\end{document}